\documentclass[11pt,a4paper]{article}
\usepackage[LGR, T1]{fontenc} % LGR is required for Greek
\usepackage[utf8]{inputenc}
\usepackage[greek, main=english]{babel}
\usepackage{authblk}

\usepackage{mathpazo}

\usepackage{amsmath}
\usepackage{amssymb}
\usepackage{amsthm}
\usepackage{mathtools}
\usepackage{xcolor}
\usepackage{hyperref}
\usepackage[left=2cm, right=2cm, top=2cm, bottom=2.5cm]{geometry}
\definecolor{red}{RGB}{235,16,16}
\definecolor{green}{RGB}{0,128,40}
\definecolor{gray}{RGB}{128,128,128}

\let\oldtext\text
\renewcommand{\text}[1]{\oldtext{\fontfamily{cmr}\selectfont #1}}

\theoremstyle{plain}
\newtheorem{theorem}{Theorem}[section]
\newtheorem{lemma}[theorem]{Lemma}
\newtheorem{claim}[theorem]{Claim}

\theoremstyle{definition}
\newtheorem{definition}[theorem]{Definition}

\newcommand{\kaliField}{\mathbb{F}}
\newcommand{\kaliIm}{\operatorname{Im}}

\newcommand{\tab}{\hspace{15pt}} 
\newcommand{\setof}[1]{\left\{ {#1}\right\}} 
\newcommand{\red}{\star}

\begin{document}

\title{\Huge \textbf{Discrete Morse Theory for Computing Filtered Cosheaf Homology}}
\author{\Large Neria Baris}
\affil{\small Mathematical Institute, University of Oxford}
\date{}
\maketitle

% % --- Epigraphs (Right-aligned and compact) ---
% \vspace{-1em} % צמצום המרווח מהכותרת

% \begin{flushright}
%     % Iliad Quote
%     \begin{minipage}{0.6\textwidth}
%         \footnotesize % פונט מעט קטן יותר לחיסכון במקום
%         \textgreek{ὡς δ᾽ ἄνεμος ἄχνας φορέει ἱερὰς κατ᾽ ἀλωὰς \\
%         ἀνδρῶν λικμώντων, ὅτε τε ξανθὴ Δημήτηρ \\
%         κρίνῃ ἐπειγομένων ἀνέμων καρπόν τε καὶ ἄχνας...} \\[0.3em]
%         \textit{"As the wind carries the chaff about the sacred threshing-floors when men are winnowing, and fair-haired Demeter separates the grain and the chaff..."} \\[0.2em]
%         \scriptsize --- Homer, \textit{The Iliad} 5.499
%     \end{minipage}
    
%     \vspace{1.5em} % מרווח מצומצם בין שני הציטוטים
    
%     % Odyssey Quote
%     \begin{minipage}{0.6\textwidth}
%         \footnotesize
%         \textgreek{ὁππότε κεν δή τοι συμβλήμενος ἄλλος ὁδίτης \\
%         φήῃ ἀθηρηλοιגὸן ἔχειν ἀνὰ φαιδίμῳ ὤμῳ, \\
%         καὶ τότε דὴ γαίῃ πήξας εὐῆρες ἐρετμόν...} \\[0.3em]
%         \textit{"I will give you a clear sign: whenever a wayfarer meets you and says you have a winnowing fan on your shoulder, then fix your oar in the earth."} \\[0.2em]
%         \scriptsize --- Homer, \textit{The Odyssey} 11.126
%     \end{minipage}
% \end{flushright}

% \vspace{1.5em}
% ---------------------------------------

\section{Introduction}

\subsection{Computational Topology and Cosheaves}

Homology and cohomology computations are the daily bread of Topological Data Analysis (TDA), and they appear in a wide variety of applications - sensor networks \cite{desilva2006coordinate}, distributed computing \cite{herlihy1999topological} and more (see the list brought in \cite{curry2015discretemorsetheorycomputing}). Persistent homology \cite{carlsson2009theory,zomorodian2005computing} generalizes classical and static homology to dynamic situations, and allows examination of how topological features change across different scales. It is used to distinguish essential features from noise, it is stable against perturbations since it is less affected by changes or noises in the original input, and it allows us to encode complex geometric information into an algebraic and invariant object \cite{cohensteiner2005stability,henselman2017matroids,yoon2018cellular,mischaikow2013morse,nanda2022tda,ghrist2014elementary}. These computations have many applications, such as image processing and computer vision \cite{clough2020topological}, computational models in biology and chemistry \cite{cang2018integration,townsend2020representation,cang2017topologynet}, and data analysis in neuroscience \cite{dabaghian2012topological}.

In the last decade, an additional and powerful topological structure has joined them – Cellular Sheaves, which allow encoding constraints, local coefficients, and algebraic information that varies from place to place over a fixed topological space \cite{curry2014sheavescosheavesapplications}. Among the promising applications, one can find machine learning \cite{ayzenberg2025sheaftheorydeepgeometry,su2025topologicaldataanalysistopological,hu2026sheaftheoretictopologicalperspectivecomplex}, spectral signal processing \cite{robinson2014topological}, and even kinematics of structures and origami \cite{cooperband2023cosheaf}.

However, both computing persistent homologies and computing sheaf cohomologies are computationally very expensive. A significant breakthrough in dealing with this barrier was made possible thanks to the application of discrete Morse theory on the source complex \cite{curry2015discretemorsetheorycomputing,nanda2012discrete,mischaikow2013morse,nanda2022tda}. This technique allows reducing the topological space to a smaller but topologically equivalent space, so that the cost of these computations decreases.

At the same time, until now these approaches have focused on the filtration \cite{yegnesh2017persistencesheaves,wei2025persistent}, i.e., change over time, of the topological space (such as a classical Vietoris-Rips filtration on a point cloud) - and not in cases where precisely the geometry of the space is fixed but the local information sitting on it and the nature of the connections between the nodes change over time. For example, a sensor network, or a communication graph. However, it should be noted that \cite{russold2022persistentsheafcohomology} proposes a unifying framework for both types of filtration, but its focus is different from ours.

Since standard persistent homology is actually a private case of persistent homology on cosheaves obtained when the cosheaf is the constant cosheaf, it begs to use cosheaves to generalize the theory to dynamic systems where the information itself changes. To save on computational cost, a necessary step is adapting the tools of discrete Morse theory to this case.

\subsection{Problem Formulation and Results}

Our problem is optimizing the computation of the persistent homology of a cosheaf filtration over a fixed cell complex. We will do this using discrete Morse theory. Our input consists of a static topological skeleton and a sequence of cosheaves defining the evolution of the information on it along the filtration. Using our algorithm, \texttt{Filtered\_Scythe}, we will locate an acyclic Morse pairing compatible with all time steps, and simultaneously compute the boundary matrices of the reduced complex (the homotopy-equivalent one), which allows an efficient and fast extraction of the persistence module.

First, let us define the parameters measuring the complexity of the complex and the filtration.

1. Let $N$ denote the cardinality of the cell complex $K$ (number of cells).

2. Let $p$ be the maximal number of cofaces for any cell, i.e., $p=\max_{x\in K}|x_{+}|$.

3. Let $D$ denote the maximal dimension of the stalks $\mathcal{F}_{i}(\tau)$ as vector spaces, which is also the maximal size of the extension matrices.

4. Assume that the pairing $\Sigma$ generated by the algorithm has $m_{k}$ critical cells of dimension $k$, and define the remaining critical complex density measure as $\tilde{m}=\sum_{k}m^{2}_{k}$ \cite{mischaikow2013morse}; put $M=\sum_k m_k$.

5. Let $\omega\approx 2.8$ be the matrix multiplication exponent.

6. Let $n$ be the number of filtration steps.

Notice that variables 1, 2, 3 and 6 are input parameters, 4 is an output parameter (depending on the efficiency of the Morse field pairing), and 5 is an algebraic property of operations on matrices. Our central result gathers the costs of scanning and Schur complement operations.

\begin{theorem} \label{thm:complexity}
Let $\mathcal{F}_{\bullet}$ be a filtration of length $n$ of cosheaves over a fixed cell complex $K$, and let $N, p, D, M, \tilde{m}$, and $\omega$ be the parameters defined above. Under the queuing rules specified below and constant-time adjacency lookup, the time complexity to build the reduced complex $K^{\Sigma}$ and its dynamically updated matrices under the \texttt{Filtered\_Scythe} algorithm is bounded by $O((n+1)NpMD^\omega+N\tilde m)$. The space complexity required to store the complex and its matrices across all filtration steps is bounded by $O((n+1)N(p+M)D^2)$.
\end{theorem}

While finding the absolute minimum number of critical cells is computationally prohibitive, the efficiency of our approach stems from the fact that even a near-optimal matching drastically lowers $\tilde{m}$. Since the time complexity in Theorem \ref{thm:complexity} scales linearly with this density measure, reducing the size of the critical complex is the primary engine for accelerating the homology computation across the filtration steps.

\textbf{The structure and design of the article, as well as the theoretical framework and methods, draw heavily upon the work presented in the groundbreaking articles \cite{mischaikow2013morse,curry2015discretemorsetheorycomputing}.
}

Chapter 2 includes the necessary background on sheaves, cosheaves, cosheaf filtration, and discrete Morse theory suitable for our needs.In Chapter 3 we will prove that a reduction of a chain complex using the critical elements is a chain itself and that it is chain-homotopy equivalent to the original chain. In the fourth chapter, we will provide an algorithm to perform the computation, prove its correctness, and calculate the memory and runtime complexity it requires, and thus prove the above Theorem \ref{thm:complexity}; the last chapter will deal with possible directions for future research. 
\section{Background}

In this part, we will review the required material regarding cellular sheaves and cosheaves \cite{curry2014sheavescosheavesapplications,curry2015discretemorsetheorycomputing,ghrist2014elementary,nanda2022tda}, persistence modules \cite{fersztand2024harder,zomorodian2005computing,ghrist2014elementary,dey2024efficientalgorithmscomplexespersistence,nanda2012discrete} and the algebraic version of discrete Morse theory \cite{forman1998morse,skoldberg2006morse,mischaikow2013morse,nanda2012discrete,nanda2022tda,chen2025algebraicmorsetheoryhomological,jollenbeck2009minimal}, and we will adapt it to our needs.

\subsection{(Co)sheaf Theory}

\begin{definition} \label{def:cellular_sheaf}
(Cellular sheaf) A cellular complex $K$ with the partial order $\leq$ "face of" defines a poset for us. A cellular sheaf $\mathcal{F}$ on it is a covariant functor $\mathcal{F}:\left(K,\leq\right)\rightarrow\text{Vect}_{\kaliField}$. That is, the following hold:

(1). For every simplex $\tau\in K$, $\mathcal{F}\left(\tau\right)$ is a vector space.

(2) For every $\tau\leq\tau'$ in $K$ there is a linear map $\mathcal{F}\left(\tau\leq\tau'\right):\mathcal{F}\left(\tau\right)\rightarrow\mathcal{F}\left(\tau'\right)$.

And the category axioms are preserved -

(1) For every $\tau\in K$, $\mathcal{F}\left(\tau\leq\tau\right)$ is the identity map on $\mathcal{F}\left(\tau\right)$

(2) For every triplet $\tau\leq\tau'\leq\tau''\in K$ it holds - 
\[
\mathcal{F}\left(\tau\leq\tau'\right)\circ\mathcal{F}\left(\tau'\leq\tau''\right)=\mathcal{F}\left(\tau\leq\tau''\right)
\]
\end{definition}

\begin{definition} \label{def:cosheaf}
\textbf{Cosheaf} is the corresponding contravariant functor, which we will denote $\mathcal{F}\left(\tau\geq\tau'\right):\mathcal{F}\left(\tau\right)\rightarrow\mathcal{F}\left(\tau'\right)$ ('extension map'). So that it also assigns a vector space to every simplex, and the composition holds in it in the following way - 
\[
\mathcal{F}\left(\tau\geq\tau'\right)\circ\mathcal{F}\left(\tau'\geq\tau''\right)=\mathcal{F}\left(\tau\geq\tau''\right)
\]
\end{definition}

Since sheaves and cosheaves are functors, the collection of all (co)sheaves over any complex $K$ creates functor categories - $\text{\textbf{Fun}}\left(K^{op},\text{Vect}_{\kaliField}\right),\text{\textbf{Fun}}\left(K,\text{Vect}_{\kaliField}\right)$. In such categories, the morphisms are defined as natural transformations \cite{maclane1998categories,riehl2016category}.

When we restrict ourselves to finite vector spaces these categories are equivalent (linearly dual). This is because we can define a dual functor between vector spaces, which reverses the directions of the morphisms. Thus, for every sheaf $\mathcal{F}$ there exists a cosheaf $\mathcal{F}^{\ast}$ obtained by point-wise taking of the dual space on each cell, and transposing the map. For every cell $\tau\in K$, we define $\mathcal{F}^{\ast}\left(\tau\right)=\left(\mathcal{F}\left(\tau\right)\right)^{\ast}$ and $\mathcal{F^{\ast}}\left(\tau\geq\sigma\right)=\left(\mathcal{F}\left(\sigma\leq\tau\right)^{\ast}\right)$ \cite{curry2014sheavescosheavesapplications}.

Since our target category is $\text{Vect}_{\kaliField}$, from the definition of a natural transformation it follows that every point-wise component must be a morphism in the target category - a linear map.

\begin{definition} \label{def:morphism}
(Morphism in the category $\text{\textbf{Fun}}\left(K^{op},\text{Vect}_{\kaliField}\right)$) In our case, for every $\mathcal{F},\mathcal{F}'\in\text{\textbf{Fun}}\left(K^{op},\text{Vect}_{\kaliField}\right)$, a morphism $\psi:\mathcal{F}\rightarrow\mathcal{F}'$ is a family of linear maps defined for every $\tau\in K$, $\psi_{\tau}:\mathcal{F}\left(\tau\right)\rightarrow\mathcal{F}'\left(\tau\right)$ such that for every $\tau\geq\sigma$ it holds - $\psi_{\sigma}\circ\mathcal{F}_{\tau\geq\sigma}=\mathcal{F}'_{\tau\geq\sigma}\circ\psi_{\tau}$.
\end{definition}

In category theory we use a monomorphism to generalize the concept of injective. A morphism $\phi:\mathcal{F}\rightarrow\mathcal{F}'$ will be called a monomorphism, if for every cosheaf $\mathcal{G}$ and for every pair of morphisms $\psi_{1},\psi_{2}:\mathcal{G}\rightarrow\mathcal{F}$ it holds that if $\phi\circ\psi_{1}=\phi\circ\psi_{2}$ then $\psi_{1}=\psi_{2}$ \cite{riehl2016category,maclane1998categories}. Since $\text{Vect}_{\kaliField}$ is an abelian category, the functor category into it, $\text{\textbf{Fun}}\left(K^{op},\text{Vect}_{\kaliField}\right)$, is also abelian. In an abelian category a morphism $\psi:\mathcal{G}\rightarrow\mathcal{F}$ is a monomorphism iff it is a monomorphism in every component, meaning iff it is point-wise injective \cite{weibel1994introduction}. We will note now that the fact that the category is abelian is what allows us, formally, to define homology later on.

\begin{definition} \label{def:monomorphism}
Monomorphism / Injective morphism ('embedding') in the category $\text{\textbf{Fun}}\left(K^{op},\text{Vect}_{\kaliField}\right)$ is a morphism where each of its components is a monomorphism. Since the morphisms in the category are families of linear maps, this is equivalent to the fact that for every cell $\tau\in K$, $\text{ker}\left(\phi_{\tau}\right)=\left\{ 0\right\} $.
\end{definition}

Armed with the basic concepts we can transfer the concept of filtration from complexes to the cosheaves themselves. Thus, we can compute persistent homology based on evolving data of a fixed structure. We will define a filtration and build chain complexes, based on \cite{nanda2022tda,curry2014sheavescosheavesapplications}.

\begin{definition} \label{def:filtration}
(Filtration of a cosheaf) A filtration of $\mathcal{F}$ is a sequence of cosheaves $\left\{ \mathcal{F}_{i}\right\} ^{n}_{i=0}$ where $\mathcal{F}_{n}=\mathcal{F}$ equipped with an injective morphism between any two consecutive indices, meaning for every $0\leq i\leq n-1$, there exists a morphism $\psi^{i}:\mathcal{F}_{i}\hookrightarrow\mathcal{F}_{i+1}$, such that we get - 
\[
\mathcal{F}_{0}\hookrightarrow\mathcal{F}_{1}\dots\hookrightarrow\mathcal{F}_{n-1}\hookrightarrow\mathcal{F}_{n}=\mathcal{F}
\]
\end{definition}

Meaning for every $\tau\in K$ and for every $i$ in the appropriate indices, $\psi^{i}_{\tau}:\mathcal{F}_{i}\left(\tau\right)\rightarrow\mathcal{F}_{i+1}\left(\tau\right)$ is an injective linear map; that is, $\mathcal{F}_{i}\left(\tau\right)$ is isomorphic to a subspace of $\mathcal{F}_{i+1}\left(\tau\right)$.

We will define a chain complex, similarly to our definition for regular sheaves.

\begin{definition} \label{def:chains}
(Chains over $K$ with $\mathcal{F}$-coefficients) For every dimension $m\geq 0$, the vector space of $m$-chains of $K$ with $\mathcal{F}$-coefficients is the product -
\[
\boldsymbol{\text{C}}_{m}\left(K;\mathcal{F}\right)=\bigoplus_{\text{dim}\left(\tau\right)=m}\mathcal{F}\left(\tau\right)
\]
\end{definition}

Meaning the direct sum of all the costalks attached to the cells of dimension $m$ \cite{curry2015discretemorsetheorycomputing}.

It is convenient to think of an element in $\boldsymbol{\text{C}}_{m}\left(K;\mathcal{F}\right)$ as a block-column vector. Each block corresponds to a specific cell of dimension $m$, and the content of the block is a vector in the space corresponding to the specific cell.

Recall the incidence coefficients of the complex $K$ \cite{nanda2022tda,ghrist2014elementary,kozlov2008combinatorial}: we assume that the cells in $K$ are ordered such that for every $m$-simplex $\tau$ there are $\tau_{-i}$ faces (?) for every $i\in\left\{ 0,\dots,m\right\} $. Then, for every pair of cells $\tau,\tau'$ we write - 
\[
\left[\tau':\tau\right]:=\begin{cases}
1 & \tau=\tau'_{-i}\ \text{for}\ i\ \text{even}\\
-1 & \tau=\tau'_{-i}\ \text{for}\ i\ \text{odd}\\
0 & else
\end{cases}
\]

Now we are able to define map operators between our chains:
\[
\partial^{\mathcal{F}}_{m}:\boldsymbol{\text{C}}_{m}\left(K;\mathcal{F}\right)\rightarrow\boldsymbol{\text{C}}_{m-1}\left(K;\mathcal{F}\right)
\]

As the transpose of the coboundary maps of the dual sheaf, matching the indices. Since we have already proven that the coboundary maps on sheaves indeed satisfy $\partial^{m}_{\mathcal{F}}\circ\partial^{m+1}_{\mathcal{F}}=0$, then from duality also $\partial^{\mathcal{F}}_{m+1}\circ\partial^{\mathcal{F}}_{m}=0$, as required \cite{weibel1994introduction,nanda2022tda} - since the transpose does not change the incidence coefficients and the transpose of a composition equals the composition of the transposes in reverse order.

The spaces are composed of direct sums so that $\partial^{\mathcal{F}}_{m}$ can be represented as a block matrix. In this matrix, each row corresponds to a cell $\sigma$ of dimension $m-1$ and each column corresponds to a cell $\tau$ of dimension $m$. The block at their intersection is the linear map 
\[
\partial^{\mathcal{F}}_{m}|_{\sigma,\tau}=\left[\tau:\sigma\right]\cdot\mathcal{F}\left(\tau\geq\sigma\right)
\]

For convenience, we will denote this expression as $F\left(\tau\geq\sigma\right)$.

So that the sequence 
\[
0\leftarrow\boldsymbol{\text{C}}_{0}\left(K;\mathcal{F}\right)\overset{\partial^{\mathcal{F}}_{1}}{\leftarrow}\boldsymbol{\text{C}}_{1}\left(K;\mathcal{F}\right)\overset{\partial^{\mathcal{F}}_{2}}{\leftarrow}\boldsymbol{\text{C}}_{2}\left(K;\mathcal{F}\right)...
\]
creates a chain complex over $\kaliField$.

Now, we can fully define a 'chain complex associated with the filtration'.

\begin{definition} \label{def:filtered_chain_complex}
(Filtered chain complex) For a filtration $\left\{ \mathcal{F}_{i}\right\} ^{n}_{i=0}$ of a cosheaf $\mathcal{F}$ over a complex $K$, we denote - 
\[
\boldsymbol{C}_{\bullet}\left(K,\mathcal{F}_{\bullet}\right)=\left\{ \boldsymbol{C}_{\bullet}\left(K,\mathcal{F}_{i}\right)\right\} ^{n}_{i=0}
\]
equipped with the family of injective morphisms $f^{i}_{\bullet}=\left\{ f^{i}_{j}:\boldsymbol{C}_{j}\left(K,\mathcal{F}_{i}\right)\hookrightarrow\boldsymbol{C}_{j}\left(K,\mathcal{F}_{i+1}\right)\right\} ^{n}_{j=0}$ where every $f^{i}_{j}$ is an injective linear map. Each such map is defined point-wise as a morphism of the cosheaf, meaning for every cell $\tau\in K$, $f^{i}_{j}|_{\tau}=\psi^{i}_{\tau}$ where $\psi^{i}$ is a morphism $\psi^{i}:\mathcal{F}_{i}\rightarrow\mathcal{F}_{i+1}$. Thus, we obtain an embedding - 
\[
f^{i}_{\bullet}:\boldsymbol{C}_{\bullet}\left(K,\mathcal{F}_{i}\right)\hookrightarrow\boldsymbol{C}_{\bullet}\left(K,\mathcal{F}_{i+1}\right)
\]
\end{definition}

Since the maps $f^{i}_{j}$ are induced from the cosheaf morphism $\psi^{i}$, and these are natural transformations commuting with the extension maps for every cell and face. Since the boundary operator is defined by the same extension maps, multiplied by scalar coefficients, the commutativity property is preserved and for every dimension $j$ it holds 
\[
\partial^{\mathcal{F}_{i+1}}_{j}\circ f^{i}_{j}=f^{i}_{j-1}\circ\partial^{\mathcal{F}_{i}}_{j}
\]

Meaning, if we look at two concrete cells - cell $\tau$ of dimension $j$ and a face of it $\sigma$ of dimension $j-1$, it holds:
\[
[\tau:\sigma]\mathcal{F}_{i+1}(\tau\geq\sigma)\circ\psi^{i}_{\tau}=[\tau:\sigma]\psi^{i}_{\sigma}\circ\mathcal{F}_{i}(\tau\geq\sigma)
\]

And we get this immediately from the fact that the morphism $\psi^{i}$ is a natural transformation. Therefore the definition is proper and this is indeed a chain complex morphism.

\begin{definition} \label{def:homology_group}
For every filtration level $0\leq i\leq n$ and for every dimension $k$, we define the $k$-th homology group of the complex as the quotient of the cycle space by the boundary space: 
$H_{k}(K;\mathcal{F}_{i})=\frac{\ker\left(\partial^{\mathcal{F}_{i}}_{k}\right)}{\kaliIm\left(\partial^{\mathcal{F}_{i}}_{k+1}\right)}$
\end{definition}

\subsection{Persistence Modules}

Now, we will turn to persistence modules. Recall that persistence modules are quite a simple object - $\left(V_{\bullet},a_{\bullet}\right)$ of vector spaces $V_{m}$ and linear maps $a_{m}$, where for every $i\leq j$ we define the transition map $a_{i\rightarrow j}$ as the composition of the maps between them \cite{nanda2022tda,zomorodian2005computing,ghrist2014elementary}. In categorical language, this is a functor from the poset category of the natural numbers with the natural order relation to the category of vector spaces.

The vector space of the persistence from $i$ to $j$ is the image of the transition map:
\[
\text{\textbf{PH}}_{i\rightarrow j}\left(V_{\bullet},a_{\bullet}\right)=\text{img}\left(a_{i\rightarrow j}\right)
\]

When we take our persistence module to be the homology groups (meaning, spaces) of the chain complex we built, we will get the persistence module of the cellular cosheaf: The inclusion maps we defined induce a sequence of homomorphisms between the corresponding homology groups. Since our filtration morphisms are chain morphisms, they map cycles to cycles and boundaries to boundaries, and therefore induce a linear map on the quotient spaces - 
\[
\eta_{i}:H_{k}(K;\mathcal{F}_{i})\rightarrow H_{k}(K;\mathcal{F}_{i+1})
\]

Now we will chain the spaces and the maps and we will obtain a homology persistence module for every dimension $k$ - 
\[
H_{k}(K;\mathcal{F}_{0})\overset{\eta_{0}}{\rightarrow}H_{k}(K;\mathcal{F}_{1})\overset{\eta_{1}}{\rightarrow}\dots\overset{\eta_{n-1}}{\rightarrow}H_{k}(K;\mathcal{F}_{n})
\]

We will denote the composition of these homomorphisms from step $i$ to step $j$ (where $i\leq j$) as the map $\eta_{i\to j}:H_{k}(K;\mathcal{F}_{i})\rightarrow H_{k}(K;\mathcal{F}_{j})$

Accordingly, we will define the space of the $k$-th Persistent Homology Group from step $i$ to step $j$ to be:
\[
PH^{i\to j}_{k}(K;\mathcal{F}_{\bullet})=\kaliIm\left(\eta_{i\to j}\right)
\]

This is exactly the subspace of homology classes in $H_{k}(K;\mathcal{F}_{j})$ which were born at step $i$ or earlier, and have not yet died (meaning, have not yet become trivial or a boundary of a chain of dimension $k+1$) by step $j$.

What do we actually have here? Let us stand on the difference between persistent homology on the space versus on the (co)sheaf. In persistent homology on the space the data determines the topology - the space changes as a function of time, and cells are added in each filtration. Here, on the other hand, what evolves is the information attached to the cells. 

Thus, while in persistent homology on the space, the boundary matrix grows and changes with time (a cell whose incidence coefficient was zero before can now receive a real value, for example) - here its 'structural skeleton' i.e - the incidence coefficients, is fixed, and what changes is the extension maps inside each block. Such, we have a matrix that as a block matrix looks identical - except the size of the blocks can grow as $i$ increases. This insight will help us in implementing the Morse algorithm in an efficient manner. 

After we have precisely defined our topological objects, we would want to use a tool in discrete Morse theory. To do this, in the next subsection we will adapt the regular definitions to this case.

\subsection{Morse Theory}

In regular Morse theory, deleting a pair of cells connected to each other preserves the homology because it is a "deformation retraction" \cite{forman1998morse}. In the context of (co)sheaves, to ensure we do not lose topological-algebraic information in the process, we also require that the map $\mathcal{F}(\tau\geq\sigma)$ be an isomorphism. Since, had the map been with a non-trivial kernel, the collapse could have produced false homology cycles; and if it were not surjective (meaning, a Cokernel exists), we would lose boundaries. The isomorphism ensures that the sub-complex we reduce is acyclic, and therefore the homology is preserved. The definitions below will rely heavily on \cite{mischaikow2013morse}.

\begin{definition} \label{def:compatible_pairing}
(Cosheaf-compatible and filtration-compatible pairing). For a cell complex $K$, a partial pairing $\Sigma$ is a collection of disjoint pairs such that $\sigma$ is a codimension-one face of $\tau$. The pairing is called acyclic if it does not contain closed paths. An acyclic partial pairing $\Sigma$ will be called $\mathcal{F}-\text{compatible}$ if the extension map $\mathcal{F}\left(\tau\geq\sigma\right):\mathcal{F}\left(\tau\right)\rightarrow\mathcal{F}\left(\sigma\right)$ is an isomorphism for every pair $\left(\sigma\vartriangleleft\tau\right)\in\Sigma$.
\end{definition}

\begin{definition} \label{def:filtration_compatible}
An acyclic partial pairing $\Sigma$ will be called $\mathcal{F}_{\bullet}-\text{compatible}$ for a cosheaf filtration $\mathcal{F}_{\bullet}$, if it is compatible for each step independently. Meaning, if for every $0\leq i\leq n$ and for every pair of cells $\left(\sigma\vartriangleleft\tau\right)\in\Sigma$ it holds that
\[
\mathcal{F}_{i}\left(\tau\geq\sigma\right):\mathcal{F}_{i}\left(\tau\right)\rightarrow\mathcal{F}_{i}\left(\sigma\right)
\]
is an isomorphism.
\end{definition}

\begin{definition} \label{def:critical_cell}
(Critical cell on an $\mathcal{F}_{\bullet}-c\text{ompatible}$ with an acyclic partial pairing $\Sigma$) Given $\Sigma$ an $\mathcal{F}_{\bullet}-\text{compatible}$ acyclic partial pairing, a cell $\tau\in K$ such that it does not participate in any pairing in $\Sigma$ will be called a critical cell.
\end{definition}

We will denote the set of critical simplices as $K^{c}$, since this is effectively the complementary set to the sets defining $\Sigma$.

Later we will use this requirement to formally prove that the reduced chain is chain-homotopy equivalent to the original chain and thus the isomorphism is preserved also on the homology quotient spaces.

\begin{definition} \label{def:gradient_path}
(Path, gradient path) A path of an $\mathcal{F}_{\bullet}-\text{compatible}$ acyclic partial pairing $\Sigma$ is a finite sequence of disjoint simplices on $K$ satisfying - 
\[
\rho=\left(\sigma_{1}\vartriangleleft\tau_{1}\vartriangleright\sigma_{2}\vartriangleleft\tau_{2}\dots\vartriangleright\sigma_{m}\vartriangleleft\tau_{m}\right)
\]
where $\left(\sigma_{i}\vartriangleleft\tau_{i}\right)\in\Sigma$. Such a path is a gradient path if additionally it holds that $m=1$ or $\sigma_{1}$ is not a face of $\tau_{m}$.
\end{definition}

In regular Morse theory, the weight of a gradient path is computed as the product of the incidence numbers, since we only refer to the order of the cells (one can think of this as the constant sheaf).In contrast to classical Morse theory, the coefficients of a (co)sheaf over a cell are not merely scalars, but vector spaces (see Definition \ref{def:chains}). Consequently, the connection between cells is represented by a linear map---a block in the boundary matrix. Therefore, the multiplicity of a gradient path is calculated as a composition of linear maps from the source space to the target space.

\begin{definition} \label{def:multiplicity}
(Multiplicity) For every filtration level $i\in\left\{ 0,\dots,n\right\} $ the multiplicity $M_{i}\left(\rho\right):\mathcal{F}_{i}\left(\sigma_{1}\right)\rightarrow\mathcal{F}_{i}\left(\tau_{m}\right)$ is a linear map defined 
\[
M_{i}\left(\rho\right)=\left(-F_i(\tau_m\geq\sigma_m)^{-1}\right)\circ F_i(\tau_{m-1}\geq\sigma_m)\circ\cdots\circ F_i(\tau_1\geq\sigma_2)\circ\left(-F_i(\tau_1\geq\sigma_1)^{-1}\right)
\]
We emphasize that there is an inverse map here and then a regular map, alternately. For $m=1$, this means $-F_i(\tau_1\geq\sigma_1)^{-1}$.
\end{definition}

\section{Morse Theory for Cosheaf Filtration}

Now, we can use these crispy definitions to construct a reduced chain complex and show that it indeed preserves the homological information of the original chain. The filtration on the cosheaf allows us to perform the reduction consistently across all steps of the persistence.

Formally, the proofs below are actually a private case of the Homological Perturbation Lemma \cite{skoldberg2006morse,chen2025algebraicmorsetheoryhomological}. This lemma discusses a situation where we have a complex with zero homology (like a paired pair of cells) connected to the rest of the complex via additional boundary operators that prevent us from deleting the homologically-insignificant complex without losing information. It states that the trivial pair can be deleted from the system, provided that all the information that was supposed to flow through it is routed anew. We compute the linear maps of the paths bypassing the pair, and add them to the boundary operator of the remaining cells. Thus, the boundary operator composed of the maps of the corresponding pairs is invertible, so that the complex defined with respect to these pairs is contractible. This is actually the algebraic description of a deformation retraction mentioned above. In the proof itself we will use the algebraic technique, as developed by \cite{forman1998morse,skoldberg2006morse,kozlov2008combinatorial,jollenbeck2009minimal}, and in particular using block matrices, inspired by the elegant proof of \cite{curry2015discretemorsetheorycomputing}, and we will use the connection to Schur complement from linear algebra, to clarify the process \cite{henselman2017matroids}.

\subsection{The Reduced Chain}

For every filtration level $i\in\left\{ 0,\dots,n\right\} $, we will define - 
\[
\boldsymbol{C}^{\Sigma}_{m}\left(K,\mathcal{F}_{i}\right)=\bigoplus_{\underset{\boldsymbol{\tau\in K^{c}}}{\text{dim}\left(\tau\right)=m}}\mathcal{F}_{i}\left(\tau\right)
\]

Meaning - the chain defined only from the product of the cosheaves on the critical simplices of dimension $m$\footnote{The choice of $\Sigma$ to denote the chain where we specifically \textbf{ignore} the elements of $\Sigma$ is quite strange, but it is easy to use, understood and became standard, so we will keep it.}. The inclusion morphisms of the filtration induce inclusion also between the reduced chain spaces, and allow us to generalize this naturally to $\boldsymbol{C}^{\Sigma}_{\bullet}\left(K,\mathcal{F}_{\bullet}\right)$.

\begin{definition} \label{def:reduced_boundary}
(Boundary operator of a reduced complex) We define the boundary operator as a family of boundary maps, where for each level $m$ we define - 
\[
\partial^{\Sigma}_{m}:\boldsymbol{C}^{\Sigma}_{m}\left(K;\mathcal{F}_{\bullet}\right)\rightarrow\boldsymbol{C}^{\Sigma}_{m-1}\left(K;\mathcal{F}_{\bullet}\right)
\]
where for every $i$, $\partial^{\Sigma}_{i,m}:\boldsymbol{C}^{\Sigma}_{m}\left(K;\mathcal{F}_{i}\right)\rightarrow\boldsymbol{C}^{\Sigma}_{m-1}\left(K;\mathcal{F}_{i}\right)$ is defined point-wise for every pair of cells of consecutive dimensions in $K^{c}$ by - 
\[
\partial^{\Sigma}_{i,m}|_{\sigma,\tau}=F_i(\tau\geq\sigma)+\sum_{\rho\in\Gamma(\tau,\sigma)}F_i(\tau_\ell\geq\sigma)\circ M_i(\rho)\circ F_i(\tau\geq\sigma_1)
\]
where $\Gamma(\tau,\sigma)$ consists of paths $\rho=(\sigma_1\vartriangleleft\tau_1\vartriangleright\cdots\vartriangleleft\tau_\ell)$ with $\sigma_1$ a codimension-one face of $\tau$ and $\sigma$ a codimension-one face of $\tau_\ell$; the direct term is zero unless $\sigma$ is a codimension-one face of $\tau$.
\end{definition}

The map is well defined, since we are guaranteed (from the complex being acyclic and finite) that there is a finite number of gradient paths, and each such multiplicity map is a legal composition of isomorphisms and extension maps. One can think of the map component between $\tau$ and $\sigma$ as the block in the reduced boundary matrix.

We will show that algebraically, reduction of a Morse pair is equivalent to calculating the Schur complement on the pairing block in the original boundary matrix \cite{henselman2017matroids}, and we will prove that this does not affect the boundary property.

\begin{claim} \label{claim:is_chain}
This is a chain. Meaning, $\partial^{\Sigma}\circ\partial^{\Sigma}=0$.
\end{claim}

\begin{proof}

We will prove by induction on the number of paired pairs in the acyclic pairing $\Sigma$. We will show that the operation of reducing a Morse pair, equivalent to computing the Schur complement of the pairing block, naturally preserves the chain property.

\textbf{Base:} If $\Sigma=\emptyset$, then all simplices are critical, $\mathbf{C}^{\Sigma}_{\bullet}=\mathbf{C}_{\bullet}$, and the boundary operator $\partial^{\Sigma}$ is exactly the original boundary operator satisfying the property.

\textbf{Step}: Let us assume the claim is true for an old pairing $\Sigma_{old}$ (we will denote it for simplicity as $\partial$), and we add a new matched pair $(\sigma\triangleleft\tau)$, where $\dim(\sigma)=p$ and $\dim(\tau)=p+1$. Since $\Sigma$ is $\mathcal{F}_{\bullet}\text{-compatible}$, at every filtration stage $\mathcal{F}_{i}\left(\tau\geq\sigma\right)$ remains an iso'.

We decompose the old chain spaces into a direct sum of the new critical cells and the paired cells, at a typical filtration level
\[
\mathbf{C}_{p+1}(K;\mathcal{F}_{\bullet})=\mathbf{C}^{\Sigma}_{p+1}(K;\mathcal{F}_{\bullet})\oplus\mathcal{F}(\tau)
\]
\[
\mathbf{C}_{p}(K;\mathcal{F}_{\bullet})=\mathbf{C}^{\Sigma}_{p}(K;\mathcal{F}_{\bullet})\oplus\mathcal{F}(\sigma)
\]

With respect to this decomposition (where the critical cells block appears first and the paired cell second), we can represent the original boundary operator $\partial_{p+1}:\mathbf{C}_{p+1}(K;\mathcal{F}_{\bullet})\to\mathbf{C}_{p}(K;\mathcal{F}_{\bullet})$ as a block matrix
\[
\partial_{p+1}=\begin{pmatrix}D & B\\
C & A
\end{pmatrix}
\]

Where the components are defined by the reduction of the operator to the corresponding sub-spaces:

$\ensuremath{A=\partial|_{\sigma,\,\tau}=F_i\left(\tau\geq\sigma\right)}$ is the invertible block representing the map from $\tau$ to $\sigma$; that is, the isomorphism.

$\ensuremath{B=\partial|_{\boldsymbol{C}^{\Sigma}_{p}\left(K,\mathcal{F}_{\bullet}\right),\,\tau}}$ is the column of blocks gathering the maps from $\tau$ to all the lower critical cells.

$\ensuremath{C=\partial|_{\sigma,\,\boldsymbol{C}^{\Sigma}_{p+1}\left(K,\mathcal{F}_{\bullet}\right)}}$ is the row of blocks gathering the maps from all the upper critical cells to $\sigma$.

$\ensuremath{D=\partial|_{\boldsymbol{C}^{\Sigma}_{p}\left(K,\mathcal{F}_{\bullet}\right),\,\boldsymbol{C}^{\Sigma}_{p+1}\left(K,\mathcal{F}_{\bullet}\right)}}$ is the sub-matrix gathering the direct boundary maps between the critical cells.

$\partial^{\Sigma}_{p+1}$ is defined as the sum of multiplicities of the gradient paths between the critical cells (Definition \ref{def:reduced_boundary}). In this induction step, the paths are divided into only two types:

Paths of length 1 - transition between critical cells, represented by $D$.

Paths of length 3 - through the new pairing - start from a critical cell in $p+1$, go down to $\sigma$ - the block $C$, go against the direction from the acyclic pairing from $\sigma$ to $\tau$, where this is an isomorphism represented by $-A^{-1}$, and go down again to a critical cell in $p$ (the block $B$). The composition of these maps from right to left gives the expression $-BA^{-1}C$.

Thus, $\partial^{\Sigma}_{p+1}$ is actually defined algebraically as a Schur complement - 
\[
\partial^{\Sigma}_{p+1}=D-BA^{-1}C
\]

Meaning, exactly the remaining operator after zeroing the row of $\sigma$ and the column of $\tau$ in the block matrix.

Since the pairing took place between dimensions $p+1$ and $p$, the rest of the spaces did not change ($\ensuremath{\boldsymbol{C}_{p+2}\left(K,\mathcal{F}_{\bullet}\right)=\boldsymbol{C}^{\Sigma}_{p+2}\left(K,\mathcal{F}_{\bullet}\right)}$ and also $\ensuremath{\boldsymbol{C}_{p-1}\left(K,\mathcal{F}_{\bullet}\right)=\boldsymbol{C}^{\Sigma}_{p-1}\left(K,\mathcal{F}_{\bullet}\right)}$). Therefore, the only operator compositions that might be non-zero are those involving the reduced dimension: The transition $p+2\to p+1\to p$, and the transition $p+1\to p\to p-1$. We will check one and the other will be symmetric to it.

The transition $p+2\to p+1\to p$:

The old boundary operator $\partial_{p+2}$ splits into a block column vector with respect to the decomposition of $\boldsymbol{C}_{p+1}\left(K,\mathcal{F}_{\bullet}\right)$ - 
\[
\partial_{p+2}=\begin{pmatrix}X\\
Y
\end{pmatrix}
\]

Where the top block represents the image into the critical cells in $p+1$. Therefore, we define the new reduced boundary operator to be: $\partial^{\Sigma}_{p+2}:=X$.

From the induction hypothesis it holds that $\partial_{p+1}\circ\partial_{p+2}=0$. We will multiply the matrices - 
\[
\begin{pmatrix}D & B\\
C & A
\end{pmatrix}\begin{pmatrix}X\\
Y
\end{pmatrix}=\begin{pmatrix}DX+BY\\
CX+AY
\end{pmatrix}=\begin{pmatrix}0\\
0
\end{pmatrix}
\]

Since $A$ is invertible, from the bottom equation we get $Y=-A^{-1}CX$. We will substitute, 
\[
DX+B\left(-A^{-1}CX\right)=0\implies\left(\underset{{\color{gray}\partial^{\Sigma}_{p+1}}}{\underbrace{D-BA^{-1}C}}\right)\underset{{\color{gray}\partial^{\Sigma}_{p+2}}}{\underbrace{X}}=0
\]
Accordingly we obtained as desired
\[
\partial^{\Sigma}_{p+1}\circ\partial^{\Sigma}_{p+2}=0
\]

The transition $p+1\to p\to p-1$ is completely symmetric, except that here we operate with a block row vector.

As stated for all other dimensions, the operators did not change and therefore vanish directly from the induction hypothesis. Hence the new operator is a valid boundary operator, and the reduced sequence constitutes a chain complex as required.
\end{proof}

Now that we have proven that $\boldsymbol{C}^{\Sigma}_{\bullet}$ is a well-defined chain complex for every step in the filtration, the final stage is to show that it preserves the homological information of the original complex.

\subsection{Chain Equivalence}

\begin{theorem} \label{thm:chain_homotopy}
(Chain-homotopy equivalence) For every step $i$ in the filtration, the Morse complex $\mathbf{C}^{\Sigma}_{\bullet}(K;\mathcal{F}_{i})$ is chain-homotopy equivalent to the original complex $\mathbf{C}_{\bullet}(K;\mathcal{F}_{i})$. Moreover, the equivalence maps commute with the filtration morphisms, and therefore the persistence module of the Morse complex is isomorphic to the original module
\[
PH^{i\to j}_{k}(\mathbf{C}_{\bullet}(K;\mathcal{F}_{i}))\cong PH^{i\to j}_{k}(\mathbf{C}^{\Sigma}_{\bullet}(K;\mathcal{F}_{i}))
\]
\end{theorem}

\begin{proof}[Proof of Theorem \ref{thm:chain_homotopy}]
Let us look first at the base case where the pairing $\Sigma$ contains a single matched pair $(\sigma\vartriangleleft\tau)$, where $\dim(\sigma)=p$ and $\dim(\tau)=p+1$. We will drop the filtration index $i$ throughout the first part of the proof for readability. We can do this since the pairing is $\mathcal{F}_{\bullet}\text{-compatible}$, and therefore the block $A$ is isomorphic for every filtration level.

First, we will construct explicitly the equivalence and homotopy maps using the block matrix representation defined in Claim \ref{claim:is_chain}.

We will define the homotopy operator $\Theta_{k}:\boldsymbol{C}_{k}\left(K,\mathcal{F}_{\bullet}\right)\to\boldsymbol{C}_{k+1}\left(K,\mathcal{F}_{\bullet}\right)$ (an operator of degree $+1$ on the original complex) such that it will act as a lift against the boundary direction through the invertible pairing. 
\[
\Theta_{p}=\begin{pmatrix}0 & 0\\
0 & A^{-1}
\end{pmatrix}
\]
And $\Theta_{k}=0$ for every dimension $k\neq p$.

Using $\Theta$ and the original boundary operator $\partial$, we will define the projection operator of the original complex onto itself as $P=Id_{\boldsymbol{C}_{\bullet}\left(K,\mathcal{F}_{\bullet}\right)}-\partial\Theta-\Theta\partial$. Where in the two affected dimensions, its matrix representation is

In dimension $p$, $\Theta_{p-1}=0$ and therefore $\Theta\partial=0$, so that 
\[
P_{p}=\begin{pmatrix}I & 0\\
0 & I
\end{pmatrix}-\begin{pmatrix}D & B\\
C & A
\end{pmatrix}\begin{pmatrix}0 & 0\\
0 & A^{-1}
\end{pmatrix}=\begin{pmatrix}I & 0\\
0 & I
\end{pmatrix}-\begin{pmatrix}0 & BA^{-1}\\
0 & I
\end{pmatrix}=\begin{pmatrix}I & -BA^{-1}\\
0 & 0
\end{pmatrix}
\]
And in dimension $p+1$, $\Theta_{p+1}=0$ and therefore $\partial\Theta=0$:
\[
P_{p+1}=\begin{pmatrix}I & 0\\
0 & I
\end{pmatrix}-\begin{pmatrix}0 & 0\\
0 & A^{-1}
\end{pmatrix}\begin{pmatrix}D & B\\
C & A
\end{pmatrix}=\begin{pmatrix}I & 0\\
0 & I
\end{pmatrix}-\begin{pmatrix}0 & 0\\
A^{-1}C & I
\end{pmatrix}=\begin{pmatrix}I & 0\\
-A^{-1}C & 0
\end{pmatrix}
\]

From the operator $P$ we will derive our equivalence maps relative to the reduced space:

The projection map $\Pi_{\bullet}:\boldsymbol{C}_{\bullet}\left(K,\mathcal{F}_{\bullet}\right)\to\boldsymbol{C}^{\Sigma}_{\bullet}\left(K,\mathcal{F}_{\bullet}\right)$ is obtained from restricting the output range of $P$ to the critical space, meaning taking the top block row
\[
\Pi_{p}=\begin{pmatrix}I & -BA^{-1}\end{pmatrix}
\]
The inclusion map $\Psi_{\bullet}:\boldsymbol{C}^{\Sigma}_{\bullet}\left(K,\mathcal{F}_{\bullet}\right)\to\boldsymbol{C}_{\bullet}\left(K,\mathcal{F}_{\bullet}\right)$ is obtained from restricting the input domain of $P$ to the critical space, taking the left block column
\[
\Psi_{p+1}=\begin{pmatrix}I\\
-A^{-1}C
\end{pmatrix}
\]

For the dimensions that were not affected by the pairing, the maps $\Pi,\Psi$ act as the identity map. To prove that these maps constitute a chain-homotopy equivalence, we must show two properties, presented in the following lemmas.

\begin{lemma} \label{lem:identity_reduced}
(Identity on the reduced complex) The composition $\Pi_{\bullet}\circ\Psi_{\bullet}$ is the exact identity map on the reduced complex $\boldsymbol{C}^{\Sigma}_{\bullet}\left(K,\mathcal{F}_{\bullet}\right)$.
\end{lemma}
\begin{proof}
We will multiply the matrices of the affected dimensions.

For dimension $p$: $\Pi_{p}\circ\Psi_{p}=\begin{pmatrix}I & -BA^{-1}\end{pmatrix}\begin{pmatrix}I\\
0
\end{pmatrix}=I$

For dimension $p+1$: $\Pi_{p+1}\circ\Psi_{p+1}=\begin{pmatrix}I & 0\end{pmatrix}\begin{pmatrix}I\\
-A^{-1}C
\end{pmatrix}=I$
\end{proof}

\begin{lemma} \label{lem:homotopy_original}
(The homotopy connection on the original complex) The composition $\Psi_{\bullet}\circ\Pi_{\bullet}$ is chain-homotopy equivalent to the identity on $\boldsymbol{C}_{\bullet}\left(K,\mathcal{F}_{\bullet}\right)$ with respect to the original boundary operator $\partial$,
\[
Id_{\boldsymbol{C}_{\bullet}\left(K,\mathcal{F}_{\bullet}\right)}-\Psi_{\bullet}\circ\Pi_{\bullet}=\partial\circ\Theta+\Theta\circ\partial
\]
\end{lemma}
\begin{proof}
We will calculate the left side:
\[
\Psi_{p}\circ\Pi_{p}=\begin{pmatrix}I\\
0
\end{pmatrix}\begin{pmatrix}I & -BA^{-1}\end{pmatrix}=\begin{pmatrix}I & -BA^{-1}\\
0 & 0
\end{pmatrix}
\]

So that 
\[
Id_{\boldsymbol{C}_{p}\left(K,\mathcal{F}_{\bullet}\right)}-\Psi_{p}\circ\Pi_{p}=\begin{pmatrix}I & 0\\
0 & I
\end{pmatrix}-\begin{pmatrix}I & -BA^{-1}\\
0 & 0
\end{pmatrix}=\begin{pmatrix}0 & BA^{-1}\\
0 & I
\end{pmatrix}
\]

We will calculate the right side. The homotopy contribution in dimension $p$ comes from $\partial\circ\Theta$ only:
\[
\partial_{p+1}\Theta_{p}=\begin{pmatrix}D & B\\
C & A
\end{pmatrix}\begin{pmatrix}0 & 0\\
0 & A^{-1}
\end{pmatrix}=\begin{pmatrix}0 & BA^{-1}\\
0 & I
\end{pmatrix}
\]

Overall the sides are exactly equal. A symmetric and identical proof holds for dimension $p+1$, where the homotopy contribution comes from $\Theta_{p}\partial_{p+1}$ and produces the matrix $\begin{pmatrix}0 & 0\\
A^{-1}C & I
\end{pmatrix}$.
\end{proof}

\begin{lemma} \label{lem:generalization_acyclic}
(Generalization to any complex acyclic pairing) There exists a full chain-homotopy equivalence for the reduction of any acyclic partial pairing $\Sigma$.
\end{lemma}
\begin{proof}
According to discrete Morse theory \cite{forman1998morse}, every acyclic partial pairing $\Sigma$ induces a partial order relation (a directed acyclic graph - $DAG$) on its elements, defined by the direction of the gradient paths. From a basic combinatorial property of finite acyclic graphs, there always exists a minimal element in them - in other words, a matched pair from which no gradient paths leave to other pairs \cite{kozlov2008combinatorial}. We will perform the reduction of the pairs by induction, step by step, according to a topological sort that will give us the minimal elements and upwards.

The requirement of acyclicity of the original pairing ensures that the reduction of a single pair (which updates the blocks $D,B,C$ of the boundary operator in the current step) only merges existing paths, and will never create a new circular gradient path (Self-loop) \cite{kozlov2008combinatorial,skoldberg2006morse}.

Since we reduce the pairs according to the minimal order, pairs that have not yet been reduced are not part of a continuation path for a pair that has already been reduced. Therefore, updating the matrices / row reduction does not change the original values of the pairing maps that have not yet been handled. This fact ensures that the block $A$ of future pairs in the queue remains unchanged ($A=F_i\left(\tau\geq\sigma\right)$), and therefore remains an invertible isomorphism throughout the entire process, which allows the continuation of the reduction.

Since a finite composition of chain-homotopy equivalences is itself a chain-homotopy equivalence \cite{nanda2022tda,munkres1984elements,ghrist2014elementary}, inductive application of Lemmas \ref{lem:identity_reduced} and \ref{lem:homotopy_original} over all the pairs in the pairing concludes after a finite number of steps. This sequence of compositions is well defined, and induces a global chain-homotopy equivalence that maps the entire original complex to the completely reduced complex.
\end{proof}

\textit{End of the proof of the theorem} \textit{(Isomorphism of the persistence modules)}:

We have shown that there exists a chain-homotopy equivalence for every single index step $i$. Now we will prove that it is preserved along the filtration.

Since the pairing was defined as $\mathcal{F}_{\bullet}$-\text{compatible}, the blocks $A_{i}$ are isomorphisms for every filtration level $i$. The equivalence maps $\Pi_{i}$ and $\Psi_{i}$ are completely defined by compositions and sums of the boundary operator and of the extension maps $\mathcal{F}_{i}(\tau\geq\sigma)$ and their inverses.

Since the cosheaf $\mathcal{F}$ is a functor, the filtration morphisms (inclusions) $\psi^{i}:\mathcal{F}_{i}\hookrightarrow\mathcal{F}_{i+1}$ are natural transformations. Therefore, the filtration morphisms commute with all the extension and boundary maps defining $\Pi$ and $\Psi$. Algebraically, we get the commutative diagram
\[
\Pi_{i+1}\circ f^{i}_{\bullet}=\tilde{f}^{i}_{\bullet}\circ\Pi_{i}
\]

where $f^{i}$ and $\tilde{f}^{i}$ are the inclusion maps of the original and reduced complex, respectively.

Since the homotopy equivalence commutes with the filtration functor, it is a filtered-chain-equivalence. As shown by \cite{mischaikow2013morse}, this equivalence induces a natural isomorphism on the persistent homology groups across all dimensions and indices. So that the persistence module of the reduced complex is isomorphic to the original module.
\end{proof}

\section{Algorithms}

It is time To reap the wheat and bring in the sheaves :) . We will use the proofs above to generate an algorithm that will allow us to reduce the chain complex in favor of optimizing the persistence module computation. Finding a perfect Morse matching is an NP-hard task \cite{Joswig2004Computing}, nearly as challenging as finding a soulmate in real life. Consequently, we shall follow the age-old advice of the elderly aunts: we settle. This compromise remains highly beneficial, as in most cases, it allows us to significantly reduce the number of critical cells $m_k$ and the resulting density measure $\tilde{m}$. We will base ourselves on the algorithm \texttt{Scythe} \cite{curry2015discretemorsetheorycomputing} and we will adapt it to our needs. The central idea is a scan similar to $BFS$ that starts from the minimal cells. The moment it finds a valid pair, it performs the Schur complement reduction on the corresponding matrices. Thus the need to build a global and heavy $DAG$ is saved from us. 

In the first stage, we will define a function that will check the compatibility of the pairing.

\begin{table}[ht]
\centering
{\bf Algorithm: }{\tt TryPairingAndCache} \label{alg:try_pairing} \\
{\bf In: }A cell $y$ and an active face of it $x$, and the effective boundary maps $\setof{F_i}$. \\
{\bf Out: } True/False, and caching matrices.  \\
{
\small
\begin{tabular}{|c|l|}
\hline
01 & {\bf initialize} Temp\_Cache = empty collection \\
02 & {\bf for each} level $i = 0$ to $n$ \\
03 & \tab {\bf if} the matrix $F_i(y \geq x)$ is NOT invertible \\
04 & \tab \tab clear Temp\_Cache \\
05 & \tab \tab {\bf return} False \\
06 & \tab {\bf else} \\
 & \tab \tab\textcolor{gray}{// For $\$-A^{1}$ Block} \\
07 & \tab \tab {\bf calculate} and add to Temp\_Cache $\left( -F_i(y \geq x)^{-1} \right)$ \\
08 & \tab {\bf end if} \\
09 & {\bf end for} \\
10 & Cached = Temp\_Cache; {\bf return} True \\
\hline
\end{tabular}
}
\end{table}

The function receives two cells that can be a Morse pair, and checks if they are $\mathcal{F}_{\bullet}-\text{compatible}$. We will require that the map be an iso' for the entire length of the filtration; notice that we save the inverse map in the cache to ease the computations later. The correctness follows directly from Definition \ref{def:filtration_compatible}.

Since checking invertibility and calculating an inverse takes $O\left(D^{\omega}\right)$ we will get $O\left(n\cdot D^{3}\right)$.

\begin{table}[ht]
\centering
{\bf Algorithm: }{\tt ReducePair} \label{alg:reduce_pair} \\
{\bf In: }A compatible pair $(x^\red,y^\red)$ that passed function \ref{alg:try_pairing}. \\
{\bf Out: } Modifies the active boundary graph and matrices (In-place).  \\
{
\small
\begin{tabular}{|c|l|}
\hline
01 & {\bf for each} cell $z \in x^\red_+ \setminus \setof{y^\red}$ in the current active graph \\
02 & \tab {\bf for each} cell $w \in y^\red_- \setminus \setof{x^\red}$ in the current active graph \\
03 & \tab \tab {\bf if} the incidence relation $w \prec z$ does not exist in the active graph \\
04 & \tab \tab \tab {\bf add} relation $w \prec z$ to the active graph \\
05 & \tab \tab {\bf end if} \\
06 & \tab \tab {\bf for each} level $i = 0$ to $n$ \\
07 & \tab \tab \tab Inverse\_Block = Cached $\left( -F_i(y^\red \geq x^\red)^{-1} \right)$ \\
08 & \tab \tab \tab {\bf replace} $F_i(z \geq w)$ by $F_i(z \geq w) + F_i(y^\red \geq w) \circ \text{Inverse\_Block} \circ F_i(z \geq x^\red)$ \\
09 & \tab \tab {\bf end for} \\
10 & \tab {\bf end for} \\
11 & {\bf end for} \\
12 & {\bf remove} $x^\red$ and $y^\red$ from the active boundary graph \\
\hline
\end{tabular}
}
\end{table}

The idea is that when $(x^{\star},y^{\star})$ is a valid pairing, we can reduce it. We create bypass paths for all the cells that went into $y^{\star}$ and had to go out of $x^{\star}$, connecting them directly instead of routing through the paired cells. Because the original boundary map goes from dimension $p+1$ to $p$, the bypass path from an upper neighbor $z$ to a lower neighbor $w$ flows through $x^{\star}$ and $y^{\star}$. To prevent dimension mismatch, the composition of linear maps must be applied in the correct algebraic order (from right to left): first mapping from $z$ to $x^{\star}$, then applying the inverse pairing from $x^{\star}$ to $y^{\star}$, and finally mapping from $y^{\star}$ to $w$. The bypass computation is performed for each filtration level $i$ separately using the inverse matrices stored in the cache.

The correctness stems from the fact that the update is exactly the Schur complement formula we proved in Claim \ref{claim:is_chain}. Since the update is performed for every $i$ in parallel, the equivalence is obtained from the Chain-Homotopy Equivalence Theorem (\ref{thm:chain_homotopy}).

The number of relevant cofaces of $x^{*}$ is bounded by $p$ (the maximal number of cells in a coboundary). The number of relevant faces is bounded by $m_{k}$. For every pair $\left(z,w\right)$ we perform two matrix multiplications, for each of the $n$ filtration levels. In total we will get $O\left(p\cdot m_{k}\cdot n\cdot D^{\omega}\right)$.

\begin{table}[ht]
\centering
{\bf Algorithm: }{\tt FilteredScythe} \label{alg:filtered_scythe} \\
{\bf In: }The original complex $K$ and the boundary maps $\setof{F_i}$. \\
{\bf Out: } The reduced complex $K^c$ and updated effective boundary matrices.  \\
{
\small
\begin{tabular}{|c|l|}
\hline
01 & {\bf initialize} empty queue: Que \\
02 & {\bf while} there are non-critical active cells left in $K$ \\
03 & \tab {\bf find} a minimal non-critical active cell $c$ in $K$ \\
04 & \tab {\bf mark} $c$ as Critical \\
05 & \tab Que.Clear() \\
06 & \tab Que.Enqueue($c$) \\
07 & \tab {\bf while} Que is not empty \\
08 & \tab \tab $y =$ Que.Dequeue() \\
09 & \tab \tab {\bf find} all active non-critical lower neighbors $x$ of $y$ in the current boundary graph \\
10 & \tab \tab {\bf if} there is EXACTLY ONE such $x$, AND TryPairingAndCache($x, y, \setof{F_i}$) is True \\
11 & \tab \tab \tab \textcolor{gray}{// We found a valid Morse pair in the active graph!} \\
12 & \tab \tab \tab Que.EnqueueAll($x_+ \setminus \setof{y}$) \\
13 & \tab \tab \tab ReducePair($x, y, \setof{F_i}$) \\
14 & \tab \tab {\bf end if} \\
15 & \tab \tab \textcolor{gray}{// Propagate search upwards regardless of pairing} \\
16 & \tab \tab {\bf if} $y$ is active: Que.EnqueueAll($y_+$) \\
17 & \tab {\bf end while} \\
18 & {\bf end while} \\
19 & {\bf return} remaining active cells (which are $K^c$) and the fully reduced matrices. \\
\hline
\end{tabular}
}
\end{table}

For the complexity bound, within each outer iteration enqueue only $c$ itself or non-critical cells of dimension $\dim c+1$, each at most once; reset the enqueue flags when Que is cleared \cite{curry2015discretemorsetheorycomputing}. The numbers $m_k$ refer to this variant, which can leave more critical cells than unrestricted queuing. Neighbor sets and $F_i$ refer to the current active boundary graph and its effective signed blocks; initialize missing blocks to zero.

The algorithm implements a greedy approach right here, similar to $BFS$, and scans the fixed complex $K$ from the minimal cells upwards. When it locates a cell $y$ with a single non-critical face $x$, it checks the compatibility of the pairing throughout the filtration. If it is an $\mathcal{F}_{\bullet}-\text{compatible}$ pairing, the algorithm fixes locally immediately, and enters the neighbors affected by the change into the queue. The cells that will remain in $K$ at the end of the scan are exactly the critical cells.

The algorithm is correct since the scan implements the inductive process of our proofs, and satisfies the necessary conditions: The first is acyclicity - the 'single non-critical face' requirement ensures that the paths are flowed in well-defined directions without cycles. In addition, the $BFS$ scan implements an implicit topological sort during the run. This sort ensures us that we do not change the extension matrices $A$ of faces belonging to pairings we have not yet checked - Lemma \ref{lem:generalization_acyclic} and the Chain-Homotopy Equivalence Theorem (\ref{thm:chain_homotopy}).

% --- REPLACE THE COMPLEXITY ANALYSIS AT THE END OF SECTION 4 ---

% --- COMPLEXITY ANALYSIS ---

To precisely evaluate the complexity, one must address the potential fill-in of the active boundary graph during reductions. While reducing a pair introduces new edges, the topological constraints of the Scythe algorithm inherently bound this density. A pairing $(x,y)$ is only executed when $x$ is the \textit{unique} non-critical face of $y$. This guarantees that any new incidence relations $w \prec z$ introduced by \texttt{ReducePair} involve lower neighbors $w$ that are \textit{already critical}. Consequently, non-critical cells never acquire new upper neighbors during the process. The number of upper neighbors $x_+$ is thus strictly bounded by the original complex parameter $p$, and $|y_-\setminus\{x\}|$ is bounded by the number of critical cells in that dimension, $m_k$.

Let $r_k$ count cancelled pairs with lower dimension $k$, and let $N_k,E_k$ count original $k$-cells and incidences between dimensions $k$ and $k+1$. Since $\sum_k r_k\leq N/2$, bypass updates cost $O((n+1)pD^\omega\sum_k r_km_k)\subseteq O((n+1)NpMD^\omega)$. The stated queuing rules allow at most $\sum_k m_kN_{k+1}\leq NM$ candidate tests, each costing $O((n+1)D^\omega)$. Graph scanning costs $O(\sum_k m_kE_k+\sum_k N_{k+1}m_k^2)\subseteq O(NpM+N\tilde m)$; selecting starting cells costs $O(NM)$, and queue and deletion work obey the same combined bound. Thus the total time is $O((n+1)NpMD^\omega+N\tilde m)$.

The space complexity must account for the dense blocks generated by the Schur complements and the storage of the original operators across all filtration steps. Since the maximum number of critical lower neighbors is bounded by the total number of critical cells, the worst-case space complexity evaluates to $O((n+1)N p D^2)$ for the initial maps, and $O((n+1)N M D^2)$ for the dynamic fill-in and caching structures, with $O((n+1)N(p+M)D^2)$ total peak storage including the input filtration maps.

\section{Future Research}

The application of discrete Morse theory on a cosheaf filtration opens the gate to many extension possibilities. Inspired by \cite{carlsson2009theory}, it is natural to examine what happens in a multi-dimensional filtration, and how one can define a consistent Morse pairing over a family of cosheaves defined on a lattice of parameters; a similar direction is a hybrid model - where both the module and the cosheaf evolve over time. Additionally the local nature of the reduction operation on the blocks suggests a parallel implementation, similarly to \cite{bauer2021ripser,bauer2014persistence}.
\small
\bibliographystyle{plain}
\bibliography{TDA_bib}

@book{ghrist2014elementary,
  title={Elementary Applied Topology},
  author={Ghrist, Robert},
  edition={1.0},
  year={2014},
  publisher={CreateSpace}
}

@book{jollenbeck2009minimal,
  title={Minimal resolutions via algebraic discrete Morse theory},
  author={J{\"o}llenbeck, Michael and Welker, Volkmar},
  number={923},
  year={2009},
  publisher={American Mathematical Society},
  series={Memoirs of the American Mathematical Society}
}

@book{munkres1984elements,
  title={Elements of Algebraic Topology},
  author={Munkres, James R.},
  year={1984},
  publisher={Addison-Wesley}
}

@book{kozlov2008combinatorial,
  title={Combinatorial Algebraic Topology},
  author={Kozlov, Dmitry},
  volume={21},
  year={2008},
  publisher={Springer Science \& Business Media},
  doi={10.1007/978-3-540-71962-5},
  series={Algorithms and Computation in Mathematics}
}

@book{maclane1998categories,
  title={Categories for the Working Mathematician},
  author={Mac Lane, Saunders},
  edition={2nd},
  series={Graduate Texts in Mathematics},
  volume={5},
  year={1998},
  publisher={Springer-Verlag}
}

@book{riehl2016category,
  title={Category Theory in Context},
  author={Riehl, Emily},
  year={2016},
  publisher={Dover Publications},
  series={Aurora: Dover Modern Math Originals}
}

@book{robinson2014topological,
  title={Topological Signal Processing},
  author={Robinson, Michael},
  year={2014},
  publisher={Springer Berlin Heidelberg}
}

@book{weibel1994introduction,
  title={An Introduction to Homological Algebra},
  author={Weibel, Charles A},
  year={1994},
  publisher={Cambridge University Press}
}

@article{cang2017topologynet,
  title={Topologynet: Topology based deep convolutional and multi-task neural networks for biomolecular property predictions},
  author={Cang, Zixuan and Wei, Guo-Wei},
  journal={PLoS Computational Biology},
  volume={13},
  number={7},
  pages={e1005690},
  year={2017},
  publisher={Public Library of Science}
}

@article{cang2018integration,
  title={Integration of element specific persistent homology and machine learning for protein-ligand binding affinity prediction},
  author={Cang, Zixuan and Wei, Guo-Wei},
  journal={International Journal for Numerical Methods in Biomedical Engineering},
  volume={34},
  number={2},
  pages={e2914},
  year={2018},
  publisher={Wiley Online Library}
}

@article{clough2020topological,
  title={A topological loss function for deep-learning based image segmentation using persistent homology},
  author={Clough, James R and Byrne, Nicholas and Oksuz, Ilkay and Zimmer, Veronica A and Schnabel, Julia A and King, Andrew P},
  journal={IEEE Transactions on Pattern Analysis and Machine Intelligence},
  volume={44},
  number={12},
  pages={8766--8778},
  year={2020},
  publisher={IEEE}
}

@article{dabaghian2012topological,
  title={A topological paradigm for hippocampal spatial map formation using persistent homology},
  author={Dabaghian, Yuri and M{\'e}moli, Facundo and Frank, Loren and Carlsson, Gunnar},
  journal={PLoS Biology},
  volume={10},
  number={1},
  pages={e1001245},
  year={2012},
  publisher={Public Library of Science}
}

@article{desilva2006coordinate,
  title={Coordinate-free coverage in sensor networks with controlled boundaries via homology},
  author={de Silva, Vin and Ghrist, Robert},
  journal={The International Journal of Robotics Research},
  volume={25},
  number={12},
  pages={1205--1222},
  year={2006},
  publisher={SAGE Publications}
}

@article{zomorodian2005computing,
  title={Computing persistent homology},
  author={Zomorodian, Afra and Carlsson, Gunnar},
  journal={Discrete \& Computational Geometry},
  volume={33},
  number={2},
  pages={249--274},
  year={2005},
  publisher={Springer}
}

@article{fersztand2024harder,
  title={Harder–Narasimhan filtrations of persistence modules},
  author={Fersztand, Marc and Jacquard, Emile and Nanda, Vidit and Tillmann, Ulrike},
  journal={Transactions of the London Mathematical Society},
  volume={11},
  number={1},
  year={2024},
  month={dec},
  doi={10.1112/tlm3.70003},
  publisher={Wiley}
}

@article{forman1998morse,
  title={Morse theory for cell complexes},
  author={Forman, Robin},
  journal={Advances in Mathematics},
  volume={134},
  number={1},
  pages={90--145},
  year={1998},
  publisher={Elsevier}
}

@article{herlihy1999topological,
  title={The topological structure of asynchronous computability},
  author={Herlihy, Maurice and Shavit, Nir},
  journal={Journal of the ACM (JACM)},
  volume={46},
  number={6},
  pages={858--923},
  year={1999},
  publisher={ACM}
}

@article{mischaikow2013morse,
  title={Morse theory for filtrations and efficient computation of persistent homology},
  author={Mischaikow, Konstantin and Nanda, Vidit},
  journal={Discrete \& Computational Geometry},
  volume={50},
  number={2},
  pages={330--353},
  year={2013},
  publisher={Springer}
}

@article{skoldberg2006morse,
  title={Morse theory from an algebraic viewpoint},
  author={Sk{\"o}ldberg, Emil},
  journal={Transactions of the American Mathematical Society},
  volume={358},
  number={1},
  pages={115--129},
  year={2006},
  doi={10.1090/S0002-9947-04-03612-4}
}

@article{townsend2020representation,
  title={Representation of molecular structures with persistent homology for machine learning applications in chemistry},
  author={Townsend, Jacob and Micucci, Christian P and Hymel, John H and Maroulas, Vasileios and Vogiatzis, Konstantinos D},
  journal={Nature Communications},
  volume={11},
  number={1},
  pages={3230},
  year={2020},
  publisher={Nature Publishing Group}
}

@article{wei2025persistent,
  title={Persistent sheaf Laplacians},
  author={Wei, Xiaotian and Wei, Guo-Wei},
  journal={Foundations of Data Science},
  year={2025},
  publisher={AIMS}
}

@article{carlsson2009theory,
  title={The theory of multidimensional persistence},
  author={Carlsson, Gunnar and Zomorodian, Afra},
  journal={Discrete \& Computational Geometry},
  volume={42},
  number={1},
  pages={71--93},
  year={2009},
  publisher={Springer}
}

@article{bauer2021ripser,
  title={Ripser: efficient computation of Vietoris-Rips persistence barcodes},
  author={Bauer, Ulrich},
  journal={Journal of Applied and Computational Topology},
  volume={5},
  number={3},
  pages={391--423},
  year={2021},
  publisher={Springer}
}

@inproceedings{bauer2014persistence,
  title={Persistence by parts: Multiscale feature detection via distributed persistent homology},
  author={Bauer, Ulrich and Kerber, Michael and Reininghaus, Jan},
  booktitle={Proceedings of the thirtieth annual symposium on Computational geometry},
  pages={477--486},
  year={2014},
  organization={ACM},
  doi={10.1145/2582112.2582121}
}

@inproceedings{cohensteiner2005stability,
  title={Stability of persistence diagrams},
  author={Cohen-Steiner, David and Edelsbrunner, Herbert and Harer, John},
  booktitle={Proceedings of the twenty-first annual symposium on Computational geometry},
  pages={263--271},
  year={2005}
}

@misc{ayzenberg2025sheaftheorydeepgeometry,
  title={Sheaf theory: from deep geometry to deep learning},
  author={Anton Ayzenberg and Thomas Gebhart and German Magai and Grigory Solomadin},
  year={2025},
  eprint={2502.15476},
  archivePrefix={arXiv},
  primaryClass={math.AT},
  url={https://arxiv.org/abs/2502.15476}
}

@misc{chen2025algebraicmorsetheoryhomological,
  title={Algebraic Morse theory via Homological Perturbation Lemma},
  author={Jun Chen and Yuming Liu and Guodong Zhou},
  year={2025},
  eprint={2404.10165},
  archivePrefix={arXiv},
  primaryClass={math.KT},
  url={https://arxiv.org/abs/2404.10165}
}

@article{cooperband2023cosheaf,
  title={A cosheaf theory of reciprocal figures: Planar and higher genus graphic statics},
  author={Cooperband, Zachary and Ghrist, Robert and Hansen, Jakob},
  journal={arXiv preprint arXiv:2311.12946},
  year={2023}
}

@misc{curry2015discretemorsetheorycomputing,
  title={Discrete Morse theory for computing cellular sheaf cohomology},
  author={Justin Curry and Robert Ghrist and Vidit Nanda},
  year={2015},
  eprint={1312.6454},
  archivePrefix={arXiv},
  primaryClass={math.AT},
  url={https://arxiv.org/abs/1312.6454}
}

@article{Joswig2004Computing,
  title = {Computing Optimal Discrete Morse Functions},
  author = {Joswig, Michael and Pfetsch, Marc E.},
  journal = {Electronic Notes in Discrete Mathematics},
  volume = {17},
  pages = {191--195},
  year = {2004},
  issn = {1571-0653},
  doi = {10.1016/j.endm.2004.03.038},
  url = {https://doi.org/10.1016/j.endm.2004.03.038}
}

@misc{dey2024efficientalgorithmscomplexespersistence,
  title={Efficient Algorithms for Complexes of Persistence Modules with Applications},
  author={Tamal K. Dey and Florian Russold and Shreyas N. Samaga},
  year={2024},
  eprint={2403.10958},
  archivePrefix={arXiv},
  primaryClass={math.AT},
  url={https://arxiv.org/abs/2403.10958}
}

@misc{hu2026sheaftheoretictopologicalperspectivecomplex,
  title={A Sheaf-Theoretic and Topological Perspective on Complex Network Modeling and Attention Mechanisms in Graph Neural Models},
  author={Chuan-Shen Hu},
  year={2026},
  eprint={2601.21207},
  archivePrefix={arXiv},
  primaryClass={cs.LG},
  url={https://arxiv.org/abs/2601.21207}
}

@misc{nanda2022tda,
  author={Nanda, Vidit},
  title={Computational Algebraic Topology (Lecture Notes)},
  howpublished={\url{https://people.maths.ox.ac.uk/nanda/cat/TDANotes.pdf}},
  note={Accessed: April 2026}
}

@misc{russold2022persistentsheafcohomology,
  title={Persistent sheaf cohomology},
  author={Florian Russold},
  year={2022},
  eprint={2204.13446},
  archivePrefix={arXiv},
  primaryClass={math.AT},
  url={https://arxiv.org/abs/2204.13446}
}

@misc{su2025topologicaldataanalysistopological,
  title={Topological Data Analysis and Topological Deep Learning Beyond Persistent Homology -- A Review},
  author={Zhe Su and Xiang Liu and Layal Bou Hamdan and Vasileios Maroulas and Jie Wu and Gunnar Carlsson and Guo-Wei Wei},
  year={2025},
  eprint={2507.19504},
  archivePrefix={arXiv},
  primaryClass={math.HO},
  url={https://arxiv.org/abs/2507.19504}
}

@misc{yegnesh2017persistencesheaves,
  title={Persistence and Sheaves},
  author={Karthik Yegnesh},
  year={2017},
  eprint={1612.03522},
  archivePrefix={arXiv},
  primaryClass={math.AT},
  url={https://arxiv.org/abs/1612.03522}
}

@phdthesis{nanda2012discrete,
  title={Discrete Morse theory for filtrations and cell complexes},
  author={Nanda, Vidit},
  year={2012},
  school={University of Pennsylvania},
  address={Philadelphia, PA},
  note={Available at ScholarlyCommons},
  url={https://repository.upenn.edu/edissertations/553}
}

@phdthesis{yoon2018cellular,
  title={Cellular Sheaves and Cosheaves for Distributed Topological Data Analysis},
  author={Yoon, Hee Rhang},
  year={2018},
  school={University of Pennsylvania},
  address={Philadelphia, PA},
  note={Supervisor: Robert W. Ghrist}
}

@phdthesis{henselman2017matroids,
  title={Matroids and Canonical Forms: Theory and Applications},
  author={Henselman-Petrusek, Gregory F.},
  year={2017},
  school={University of Pennsylvania},
  address={Philadelphia, PA},
  note={Supervisor: Robert W. Ghrist}
}

@misc{curry2014sheavescosheavesapplications,
      title={Sheaves, Cosheaves and Applications}, 
      author={Justin Curry},
      year={2014},
      eprint={1303.3255},
      archivePrefix={arXiv},
      primaryClass={math.AT},
      url={https://arxiv.org/abs/1303.3255}, 
}

\end{document}